\documentclass[a4paper,11pt,reqno]{amsart}
\usepackage{amsthm}
\usepackage{amsmath}
\usepackage{enumerate}
\usepackage{amsfonts}
\usepackage{amssymb}
\usepackage{fullpage}
\usepackage{amsmath,amscd}
\usepackage{stmaryrd}
\usepackage{graphicx}
\usepackage{array}
\usepackage{amsmath,amssymb,amsfonts,dsfont}
\usepackage[utf8]{inputenc} 
\usepackage[english]{babel}
\usepackage[T1]{fontenc} 
\usepackage{graphicx}
\usepackage{float}
\usepackage{pdfpages}
\usepackage[a4paper, margin = 3cm, bottom = 3cm]{geometry}
\usepackage{ifpdf}
\usepackage{marginnote}
\usepackage{hyperref}	
\hypersetup{pdfborder=0 0 0, 
	    colorlinks=true,
	    citecolor=black,
	    linkcolor=blue,
	    urlcolor=red,
	    pdfauthor={Guillaume Tahar}
	   }
\newtheorem{thm}{Theorem}[section]
\newtheorem{cor}[thm]{Corollary}
\newtheorem{prop}[thm]{Proposition}
\newtheorem{lem}[thm]{Lemma}
\theoremstyle{definition}
\newtheorem{defn}[thm]{Definition}
\theoremstyle{remark}
\newtheorem{rem}[thm]{Remark}
\theoremstyle{definition}

\theoremstyle{definition}

\theoremstyle{definition}

\numberwithin{equation}{section} 
\title{Triangulations of branched affine surfaces}
\author{Guillaume Tahar}
\address[Guillaume Tahar]{Faculty of Mathematics and Computer Science, Weizmann Institute of Science,
Rehovot, 7610001, Israel}
\email{tahar.guillaume@weizmann.ac.il}
\date{November 29, 2019}
\keywords{Flip, Triangulated Surface, Flat structure, Branched Affine structure, Dilation surface, Stable curves}
\begin{document}
\begin{abstract}
A branched affine structure on a compact topological surface with marked points is a complex affine structure outside the marked points. We give a proof of an unpublished foundational theorem of Veech, stating that any branched affine surface can be decomposed into affine triangles and some annulus-shaped cylinders. Then, we prove that any pair of such decompositions can be connected by a chain of flips. As a first step toward a compactification of the moduli spaces of branched affine structures, we introduce invariant $\alpha$ of a branched affine surface that controls degeneracy of the structure. Finally, we consider some examples of compactification of spaces of branched affine surfaces by stable curves.\newline
\end{abstract}
\maketitle
\setcounter{tocdepth}{1}
\tableofcontents

\section{Introduction}

For a compact orientable topological surface with marked points, a \textit{branched affine structure} is an atlas of charts to the complex plane defined on the complement of the marked points such that the transition maps are affine maps $x \mapsto ax+b$ with $a \in \mathbb{C}^{*}$ and $b \in \mathbb{C}$, see \cite{G1}. Many branched affine surfaces appear as polygons where pairs of side are arbitrarily identified (without necessarily respecting directions and lengths).\newline
Any notion that makes sense in the affine complex plane is well-defined in a branched affine surface. In particular, we have straight lines, angles. In that way, they are generalizations of translation surfaces whose directional flows have been deeply studied, see \cite{Zo}. Geometric triangulations of similar structures have already been studied.\newline
Concerning flat surfaces with conical singularities (corresponding to the case where the coefficients of transition maps are required to be of modulus one), there is already a proof that they can be triangulated and that the flip graph of their triangulations is connected, see \cite{Ta}. However, there is no globally defined notion of length for branched affine structures and some topological arcs do not have geodesic representatives. The proof does not generalize.\newline
In the context of branched affine surfaces, there is an unpublished proof of Veech (see \cite{V}) that involves maximal embedded disks and Delaunay triangulations. The main ideas have been explained in \cite{DFG}. We would like to present an elementary proof that branched affine surfaces can be essentially triangulated (Theorem 1.1) and also adress the issue of flips (Theorem 1.2). A flip is a local operation where the diagonal of a convex quadrilateral is replaced by the other diagonal to get a slightly different triangulation. The topological problem of connecting two different topological triangulations of a topological surface by a chain of flips has been solved in \cite{FST}.\newline
This paper is also an opportunity to introduce several technical results (Lemmas 2.4, 3.1 and 3.3 in particular) that could be helpful by themselves in any further study of branched affine or dilation surfaces.\newline

\begin{thm}
A branched affine surface $X$ can be triangulated if and only if there is no hyperbolic cylinder of angle at least $\pi$ in $X$.\newline
\end{thm}

\begin{thm}
For a triangulable branched affine surface, any pair of geometric triangulations can be connected by a chain of flips.\newline
\end{thm}

A motivation to study triangulations of branched affine structures is about compactification. Degeneration of a translation structure corresponds to the shrinking of a loop or a geodesic segment connecting two singularities, see \cite{BCGGM}. Since we do not have a good notion of length, we have to work with angles. Instead of considering degeneration of edges, compactification of branched affine structure will need taking into account degeneration of triangles. The more we know about the geometry of the space of the triangulations of a given branched affine surface, the more we will be able to understand how the structure degenerates independantly from the arbitrary choice of the triangulation. For this purpose, we introduce invariant $\alpha$ of a branched affine surface that is measure of how far is the surface from being tiled by equilateral triangles. Bounded values of this invariant define compact parts of the moduli space of such structures (Proposition 4.2). Finally, we consider some examples of compactifications of spaces of branched affine structures.\newline

\section{Branched affine structures and geometric triangulations}

\subsection{Singularities and trajectories}

As said in the introduction, branched affine structures are affine structure structures outside the marked points. We also require that the singularities are \textit{conical singularities}. Before defining conical singularities, we have to study the holonomy character of the structure.\newline

In a branched surface $X$ punctured at the marked points $\Lambda$ (we always assume $\Lambda \neq \emptyset$), every path can be covered with charts of the atlas. The transition map between the first chart and the last chart is an affine map. Its linear part is well-defined up to conjugacy. The \textit{linear holonomy} of paths is a topological invariant given by a group morphism $\rho: H_{1}(X \setminus \Lambda, \mathbb{Z}) \longrightarrow \mathbb{C}^{*}$. If the linear holonomy of any loop is real, the branched affine surface is a \textit{dilation surface}. If the linear holonomy of any loop is trivial, it is a \textit{translation surface}.\newline

Another topological invariant of paths is the index. Let $\gamma$ be a simple closed curve in a branched affine surface. In any chart, $\theta'(t)$ is the argument of tangent vector at $\gamma(t)$. We define $\Theta(\gamma)=\dfrac{1}{2\pi}\int_{0}^{T} \theta'(t)dt \in \mathbb{R}$ to be the \textit{index} of $\gamma$. Obviously, the index of a loop and the argument of its linear holonomy are equal modulo $2\pi$.\newline

The neighborhood of such a conical singularity is constructed starting from an infinite cone of angle $\Theta(\gamma)$ and a ray starting from the origin of the cone. Then, we identify the right part of the ray with the left part of the ray with a homothety ratio of $|\rho(\gamma)|$. Thus, we get the local model of a conical singularity with the adequate invariants.\newline

The conical singularities satisfy an analog of Gauss-Bonnet formula, proved as Proposition 1 in \cite{G2} in the context of dilation surfaces but the proof is the same for branched affine surfaces. We draw a geometric triangulation whose set of vertices includes the conical singularities. Euler characteristic and the fact that the sum of the angles of any affine triangles is $\pi$ proves the first formula. The second formula is just the computation of the linear holonomy of some loop whose homology is trivial.

\begin{prop}
In a branched affine surface of genus $g$ with conical singularities $A_{1},\dots,A_{n}$ of angle $\theta_{i}$ and dilation ratio $\lambda_{i}=\mid \rho(\gamma_{i})\mid$, we have:\newline
(i) $\sum_{i=1}^{n} (\theta_{i}-2\pi) = 2\pi(2g-2)$;\newline
(ii) $\sum_{i=1}^{n} log(\lambda_{i}) = 0$.\newline
\end{prop}

Proposition 2.1 implies in particular that any branched affine surface of genus zero has at least three singularities.\newline

Path with a locally constant direction are (geodesic) \textit{trajectories}. They are straight lines.\newline

Though we do not have a globally defined notion of length, there is natural distinction between finite and infinite trajectories. A trajectory is \textit{critical} if it hits a conical singularity. Some trajectories may be self-intersecting.

\begin{defn} A \textit{saddle connection} is a geodesic segment (without self-intersection outside the ends) whose ends are conical singularities.
\end{defn}

Another interesting class of trajectories is formed by closed geodesics.

\begin{defn} A \textit{closed geodesic} $\gamma$ is a locally straight simple loop. Depending on their linear holonomy, they can be flat if $|\rho(\gamma)|=1$ or hyperbolic otherwise.
\end{defn}

Since there is no self-intersection, the linear holonomy along a closed geodesic is real. Directions are preserved. Continuation shows clearly that closed geodesics describe cylinders bounded by saddle connections. Flat cylinders are just parallelograms with a pair of opposed sides that are identified. Hyperbolic cylinders (filled by closed geodesics with nontrivial linear holonomy) are portions of annuli where the two concentric arcs of circle are identified. Their important invariant is the total angle of the portion of annulus (it can be given by the index of a path connecting the two boundaries of the cylinder), see \cite{G2} for a detailed presentation.\newline

Dynamics of branched affine structure can be complicated. Some infinite trajectories may accumulate on Cantor sets. However, we will only use very basic notion of dynamics. In particular, in hyperbolic cylinders, if an entering trajectory belongs to the same direction as some closed geodesics of the cylinder, then it will accumulate on one of them (limit-cycle). Indeed, since they have the same direction, they cannot cross each other and the trajectory cannot "turn back" and leave the cylinder by the side through which it entered.

The following lemma is a refinement of Lemma 7 of \cite{DFG}. A version concerning affine immersion of half-planes is crucial in the proof of Veech theorem on triangulations of dilation surfaces.

\begin{lem}
For a branched affine surface $X$ without self-intersecting trajectories, the image of any affine immersion $f$ of an open cone $C$ of angle $\theta$ in $X$ contains a hyperbolic closed geodesic that belongs to a hyperbolic cylinder of angle at least $\theta$. Every ray of the cone accumulates on some closed geodesic of the cylinder.
\end{lem}

\begin{proof}
The image of $C$ only contains regular points. We first show that there is a closed geodesic in the image of the cone. We consider the image $\gamma$ of a ray of the cone. Either it is a closed or it is infinite (we excluded self-intersections). In the latter case, we consider $x$ one of its accumulation points. Let $D$ be a small disk around $x$ for the affine structure of X. The intersection of $\gamma$ with $D$ is formed by infinitely many segments accumulating on at least one diameter of the disk. The translations of cone $C$ along $\gamma$ remain inside $C$ so there are cones starting from any point of $\gamma \cap D$ in $f(C)$. These small cones overlap and the immersion is clearly noninjective. The cone is convex so for any pair of points $a,b \in C$ such that $f(a)=f(b)$, segment $[a;b]$ is a closed geodesic (there is no self-intersection). The procedure we used provides closed geodesics in directions that are arbitrarily close to that of $\gamma$. This means in particular that the straight line continuing $[a;b]$ in $C$ cuts it into two infinite components.\newline
The continuation of any such closed geodesic $\alpha$ is a a cylinder bounded by singular points. If the closed geodesic is flat (homothety ratio equal to $1$), then the cylinder can be continuated inside the image of the cone to infinity. Since the surface is compact, there is no infinite flat cylinder. Besides, it cannot close up and form a torus because it would be a connected component without boundary and it should contain at least one conical singularity. Therefore, $\alpha$ is hyperbolic and its cylinder is a part of an annulus. We consider trajectories starting from a point of $\alpha$. For an open interval of directions around that of $\alpha$, trajectories accumulate on closed geodesics of the same cylinder. In particular, they cannot hit a singularity. For a sequence of directions converging to one of the two ends of the interval, the trajectory hits a singularity of the boundary of the cylinder. Since trajectories starting from $\alpha$ in directions covered by cone $C$ are infinite, the directions of closed geodesics of the cylinder form an interval of angle at least $\theta$. Preimages of the closed geodesics of the cylinder intersect the two sides of cone C. Therefore, any ray of $C$ eventually enters in the cylinder. For any ray, there is a continuous family of parallel trajectories, one of which being a closed geodesic of the cylinder. Every other trajectory of this family accumulates on this geodesic. Otherwise it would hit a singularity (and there is no conical singularity in the image of the cone). This ends the proof.
\end{proof}

\subsection{Geometric triangulations}

In the proofs that follows, we will need a notion of branched affine surface with boundary.

\begin{defn} A \textit{branched affine surface with boundary} is a topological surface $X$ with boundary with at least one marked point on each boundary component and branched affine structure in its interior such that any boundary arc is a saddle connection. We define the \textit{singular locus} $Sing(X)$ of $X$ as the union of the conical singularities and the boundary.
\end{defn}

The geometric triangulation we are interested in are a special class of topological triangulations.

\begin{defn} A topological triangulation of a topological surface with marked points and boundary (possibly empty) is a maximal family of topological arcs connecting marked points an such that they do not intersect themselves or each other in their interior. The arcs cuts out the surface into ideal triangles (vertices and edges may be not distinct). Ideal triangles can be self-folded (two sides are identified).
\end{defn}

The definition of a geometric triangulation in the framework of branched affine structures is straightforward.

\begin{defn} For a given branched affine surface, a geometric triangulation is a topological triangulation whose edges (including the boundary) are saddle connections and whose vertices are conical singularities (every conical singularity should be a vertex of the triangulation).
\end{defn}

We should recall that there are no bigons in flat surfaces (since there are simply connected, they would be translation surfaces and there is no such object). As a consequence, there is at most one saddle connection in each isotopy class of topological arcs of the surface punctured at the conical singularities.

Self-folded triangles often appear in triangulations of branched affine surfaces. They characterize conical singularities with an angle smaller than $\pi$.

\begin{lem}
Every conical singularity with an angle smaller than $\pi$ is the distinguished vertex of a self-folded triangle (with possibly several singularities in the opposite side). Conversely, the distinguished vertex of any self-folded triangle is a conical singularity of angle smaller than $\pi$.
\end{lem}

\begin{proof}
For any critical trajectory $\gamma$ starting from the conical singularity $A$ of angle $\theta < \pi$, we can locally draw a family of concentric broken geodesics whose breaking point is on $\gamma$. At any breaking point, the interior angle is $\pi-\theta$ and the exterior angle is $\pi+\theta$. Every broken geodesic defines with the corresponding portion of $\gamma$ a self-folded triangle. The maximal family of broken geodesics is finite because the underlying surface is compact. There are two possibilities. Either the family of broken geodesics hits another conical singularity or get some self-intersection.\newline
We prove that this second case cannot happen. We consider the first broken geodesic family for which this self-intersection can happen. If it involves regular points, then by minimality it is a self-tangency between entire portions of the curve. Since the exterior angle at the breaking point at $\pi+\theta$, there is necessarily some self-intersection between the breaking point and one other point of the curve. This contradicts minimality. Therefore, the maximal family of broken geodesics is bounded by singular curve formed by geodesics segments with a breaking point and at least one conical singularity $B$. Since the family of broken geodesics defines a self-folded triangle, there is inside it a saddle connection $[AB]$ that we can take as the locus of breaking points of a new family of broken geodesics. Processing the same reasoning, we get a new self-folded triangle whose vertices are this time conical singularities.
\end{proof}

\begin{rem} The self-folded triangle obtained for a conical singularity in Lemma 2.8 is by no means unique. It clearly depends on the choice of the saddle connections to which the breaking points of the broken geodesics belong to.
\end{rem}

\begin{lem}
In any branched affine surface with boundary that is not a triangle or a hyperbolic cylinder of angle at least $\pi$, there is at least one interior saddle connection.
\end{lem}

\begin{proof}
We assume by contradiction there is a counterexample $X$ of minimal genus, number of connected components of the singular locus (union of the singular points and the boundary saddle connections) and number of boundary saddle connections (in this lexicographic order). Since polygons are also translation surfaces, we already know they are triangulated, see \cite{Ta}. Besides, cylinders are bounded by saddle connections so we can assume there is no closed geodesic in our counterexample. In the case where the whole surface is a cylinder and its boundary is the boundary of the surface, there are two cases. If it is a flat cylinder or a hyperbolic cylinder of angle smaller than $\pi$, then it can be built as a trapezoid and has clearly a diagonal. It can be triangulated. If it is a hyperbolic cylinder of angle at least $\pi$, there is no straight line between the two ends of the cylinder because it always has the same direction of a closed geodesic of the cylinder and thus cannot cross it (the linear holonomy of the main loops of the cylinder is real so directions are well-defined). There is no closed geodesic in the minimal counterexample.
\newline
Then, we prove that there is no broken geodesic that is not homotopic to a component of the singular locus. A broken geodesic $\gamma$ is formed by a simple closed geodesic segment whose both ends are a turning point $A$. Clearly, it cannot be homotopically trivial. Therefore, $X \setminus \gamma$ has a geometric triangulation. It induces a geometric triangulation of $X$ where $A$ is a vertex and $\gamma$ is an edge. Point $A$ is a regular point so the total angle of its corners is $2\pi$. We consider the quadrilateral formed by the two triangles around an edge whose two ends are $A$ ($\gamma$ for example). This edge $\alpha$ is a diagonal of this quadrilateral. If it is convex, then at least one of the two other vertices outside the ends of $\alpha$ is $A$ because otherwise the other diagonal would be an interior saddle connection. Therefore, either one of the two triangles has $A$ as its three vertices or the quadrilateral is nonconvex. In both cases, the total angle of $A$ in this quadrilateral is strictly bigger than $\pi$.\newline
Consequently, if there are two edges whose both edges are $A$, the quadrilaterals they form share one triangle $T$ whose three vertices are $A$. There are also three triangles where two vertices are $A$ (the neighbors of this triangle). The other triangles have only one vertex that is $A$. There is no triangle whose three edges do not meet $A$ because its edges would be saddle connection and the surface would be just that triangle. Clearly, the surface is of genus zero with a singular locus of three connected components (one in each connected component of the complement of triangle $T$ in $X$). Loop $\gamma$ is thus homotopic to a connected component of the singular locus.\newline
If there is only one edge whose both ends are $A$, then the quadrilateral around it is nonconvex and its two triangles are the only one with two vertices that are $A$. For the others, $A$ is just one among three vertices. The surface is thus a sphere with a singular locus of two connected components (one in each of connected component of the complement of $\gamma$ in $X$). Consequently, in our minimal counterexample $X$, every broken geodesic is homotopic to a connected component of the singular locus.\newline
In fact, we can even prove that for a component $C$ of the singular locus, for any critical trajectory $\beta$ starting from $C$ that intersects itself forming a broken geodesic $\gamma$, this loop $\gamma$ is homotopic to $C$. Otherwise, we would cut along $\beta$ and $\gamma$. We would get a surface with strictly less connected component of the singular locus. The reconstruction of a true triangulation on the surface would be made as previously. The only case that is left for a critical trajectory with self-intersection is if it homotopic to its starting connected component that is formed by at most one saddle connection (or just one conical singularity). Otherwise, this process would provide a surface with strictly fewer boundary saddle connections.\newline

Then, we prove that there is no finite critical trajectory (without self-intersection) that is not homotopic to a portion of a connected component of the singular locus. Finite critical trajectories are just segments that cross a boundary saddle connection. If there is such trajectory $\alpha$ that crosses a boundary saddle connection in a regular point $A$, we can proceed as previously and get a geometric triangulation with $A$ as a vertex and $\alpha$ as an edge. The sum of the angle of the corners of $A$ is $\pi$. Thus, there is no triangle whose three vertices are $A$. If there is an interior edge such that both ends are $A$, then its a diagonal of a quadrilateral such that the two other vertices are the true conical singularities of $X$. They cannot be joined by another diagonal because it would be an interior saddle connection. This implies that the quadrilateral is nonconvex and the sum of the angles of the two corners corresponding to $\pi$ is strictly bigger than $\pi$. Therefore, $A$ appears as a vertex one time in every triangle (it is always a vertex of any triangle because otherwise there would be interior saddle connections). The opposite side of $A$ in every triangle is a boundary saddle connection. This implies that $X$ is fact a polygon and we already know that they are triangulable (Main Theorem of \cite{Ta}).\newline

We then consider two cases for each connected component of the singular locus depending on the existence of a self-intersecting critical trajectory. If there is no such trajectory, every critical trajectory is either finite (crosses a boundary saddle connection) or infinite (without self-intersection). Lemma 2.3 proves that existence of any cone of infinite trajectories implies existence of a closed geodesic. Therefore, for any corner, directions of infinite trajectories is a nowhere dense set whereas directions of finite trajectories are an open set. Any finite trajectory from a corner $A$ is homotopic to a portion of the connected component of the singular locus so it cuts out the surface into two components one of which is a polygon. Every trajectory starting from $A$ and beginning in this polygon is also finite (otherwise it crosses the first trajectory and there is a bigon, which is a contradiction for branched affine surfaces). Therefore, in every corner, there is at most one infinite trajectory that separates two pencils of finite trajectories. The points accessible by a ray from a corner are an open set so there it cannot be a union of two triangles and one infinite ray. The infinite trajectory thus cannot exist. Every critical trajectory is finite and ends in the same boundary component (by continuity). The trajectories from the corner are form thus a triangle where the ends of the two neighboring boundary saddle connections belong to the sides. Either they are in the interior of the sides (and there is an interior saddle connection) or they are the two other vertices of the triangle that form the whole surface.\newline
In the last case, every connected component of the singular locus is a singular point or a closed saddle connection. It has a self-intersecting trajectory defining a broken trajectory homotopic to it. If the component is a singular point, then the self-intersecting trajectory cuts out a bigon formed by both sides of the trajectory before the first self-intersection point and then the broken geodesic. Therefore the component is a closed saddle connection. The same construction gives a triangle whose vertices are the conical singularity (for two vertices) and the turning point of the broken geodesic (for the third one). This implies in particular that the angle of the conical singularity $\theta$ is strictly smaller than $\pi$.\newline
Every critical trajectory is either infinite or self-intersecting (no nontrivial finite trajectory). Directions of infinite trajectories are a nowhere dense closed set (because of Lemma 2.3 that proves existence of a closed geodesic in the image of every affine immersion of a cone) whereas directions of self-intersecting trajectories are an open set (we can perturbate slightly). Every critical self-intersecting trajectory draw a triangle and there is a corner in this triangle where every critical trajectory is self-intersecting. Therefore, there is at most one infinite critical trajectory starting from this conical singularity (if there were two of them, there would be a self-intersecting trajectory between them and we would get a contradiction). Each of the two intervals of directions define a continuous family of broken geodesics. These geodesics are disjoint and form a topological cylinder one end of which is the boundary component. This family is clearly finite ($X$ is compact). Just like in the proof of Lemma 2.8, the other boundary is either singular or self-intersecting. In this last case, we would be able to draw some figure-eight shaped curve whose two branches would represent an homology class that is a half of the class of the broken geodesics of the family. This is impossible. Therefore, the boundary is a singular curve formed by geodesic segments, singularities and a breaking point (which may be a singularity). The interior angle at the breaking point is $2\pi-\theta$ and $\pi$ at the other singularities. Since we are in a case where every conical singularity has an angle strictly less than $\pi$, we get the wanted contradiction.
\end{proof}

The immediate consequence of Lemma 2.10 is that cutting along saddle connections, we get smaller and smaller connected components of the surface. At the end of this finite process (the maximal number of nonisotopic nonintersecting topological arcs is finite for topological reasons), irreducible parts are either triangles or hyperbolic cylinders of angle at least $\pi$. This proves Theorem 1.1.\newline

\section{Flips between geometric triangulations}

Any maximal system of saddle connections in a branched affine surface decomposes it into a union of triangles and affine cylinders of angle at least $\pi$. The following lemma proves that long enough cylinders are disjoint so there is no arbitrary choice. The sharp bound $\frac{\pi}{2}$ is not necessary here but could be useful elsewhere.

\begin{lem} In a branched affine surface, hyperbolic cylinders of angle at least $\frac{\pi}{2}$ are disjoint.
\end{lem}

\begin{proof}
We consider a regular point $x$ that belongs to the interior of two different cylinders $C$ and $D$ of angle at least $\frac{\pi}{2}$. This means that $x$ is the intersection of two closed geodesics of $C$ and $D$. Therefore, any closed geodesic of $C$ intersects any closed geodesic of $D$. Since the angles of the hyperbolic cylinders are at least $\frac{\pi}{2}$, there is a direction $\theta$ around $x$ that is a direction of closed geodesic (or the chain of saddle connections that bounds the cylinder) for both $C$ and $D$. One of the two trajectories starting from $x$ in direction $\theta$ accumulates of a closed loop (possibly singular) of cylinder $C$. This is the same for cylinder $D$. There are two cases, either $C=D$ or the union of the two trajectories has two limit cycles in different cylinders. These two limit cycles intersect each other, therefore, the trajectory starting from $x$ in direction $\theta$ intersects itself, which is impossible. Therefore, there is no point $x$ that belongs to the interior of both cylinders $C$ and $D$.
\end{proof}

\begin{cor}
Any branched affine surface canonically decomposes into a triangulable locus and the union of this hyperbolic cylinders of angle at least $\pi$. The triangulable locus is a branched affine surface with boundary.
\end{cor}

The main source of difficulty for geometric proofs in branched affine surface is that there is not always a geodesic representative in a given homotopy class of topological arcs (for example, a topological arc between the two boundary components of a hyperbolic cylinder of angle at least $\pi$). However, it works in the triangulated locus of the surface. This will be crucial for generalizing the results of \cite{Ta} in a framework where there is no length.

\begin{lem}
In the triangulated locus of a branched affine surface, every nontrivial topological arc has a geodesic representative (chain of saddle connections or closed geodesics).
\end{lem}

\begin{proof}
Every nontrivial topological arc $\gamma$ that belongs to the triangulated locus $T$ of a branched affine surface lifts to a topological arc $\overline{\gamma}$ of the universal cover of $T$ where it passes through finitely many triangles. The union of these affine triangles is a polygon (simply connected surface with a unique boundary component) where a notion of length can be globally defined. Arc $\overline{\gamma}$ has a geodesic representative whose projection on $T$ is a geodesic representative of $\gamma$.
\end{proof}

In the following, we consider only triangulable branched affine surfaces. We study the relations between geometric triangulations of the same surface. For this purpose, we introduce the flip as an elementary move that changes only one edge.

\begin{defn} A flip is a transformation of a topological triangulation that removes an edge that is a diagonal of a quadrilateral and replaces it by the other diagonal. A flip between two geometric triangulations of a branched affine structure exists only if the quadrilateral is convex.
\end{defn}

The following theorem proves that the flip graph of geometric triangulations of a branched affine surface is connected. Using Lemma 3.3, we are able to follow the proof used in \cite{Ta} in the context of flat surfaces (where a length is globally defined). Therefore, we only give a sketch of the proof.

\begin{proof}[Proof of Theorem 1.2]
The number of interior edges of any geometric triangulation of a branched affine surface with boundary does not depend on the triangulation (we have just to sum the angles at the conical singularities). We assume by contradiction there is a counterexample $X$ with a minimal number of interior edges and two geometric triangulations $S$ and $T$ that cannot be joined by a sequence of flips. There is no interior edge that belongs to both $S$ and $T$ because otherwise we could cut along this edge and get a simpler surface where the induced triangulations are joined by a chain of flips. By forgetting the cut, we would have a sequence of flips between $S$ and $T$. For the same reasons, for every edge $x$ of $S$ and every edge $y$ of $T$, $x$ and $y$ always have a nontrivial intersection. Indeed, if there were such a pair of edges, we could complete $x$ and $y$ to get a new geometric triangulation $Z$ (Lemma 2.7). Triangulations $S$ and $T$ would have each a common edge with $Z$. There would be chain of flips from $S$ to $Z$ and then from $Z$ to $T$.\newline
For every pair of intersecting saddle connection $\alpha$ and $\beta$, we can desingularize the intersection points and get topological arcs that do not intersect $\alpha$ and $\beta$ or with each other. Thanks to Lemma 3.3, these desingularized arcs have geodesic representatives. Therefore, if there are several connected components in $Sing(X)$, any geometric triangulation has some edges relating them. Such edges for $S$ and $T$ intersect each other. Desingularization provides new saddle connections that could be completed to form another geometric triangulation $Z$ that would be an intermediate step in a chain of flips between $S$ and $T$. Therefore, the singular locus of our minimal counterexample is connected. For the same reasons, a genus different from zero would imply existence of saddle connections whose free homotopy class is nontrivial and desingularization would provide in the same way a chain of flips. This implies that $X$ is a polygon (genus zero with a connected singular locus). Branched affine polygons have trivial holonomy. They are just translation surfaces. Connectedness of their flip graph has already been proved in \cite{Ta}.\newline
\end{proof}

\section{Invariant $\alpha$ of branched affine surfaces}

For a compact topological surface of genus $g$ with $n$ conical singularities, the moduli space of triangulable branched affine structures is a complex-analytic orbifold of dimension $4g-4+2n$. Indeed, computation of Euler characteristic shows that the surface is formed by $4g-4+2n$ affine triangles. Each of them is characterized by a complex parameter and there is no constraint concerning the gluing of sides. We denote by $T_{g,n}$ this moduli space.\newline

In moduli spaces of translation surfaces of normalized area, the degeneracy of structures is controlled by the systole: the length of the smallest saddle connection. Parts where the systole is bounded by below are relatively compact in the moduli space. Compactification of the moduli spaces involve collision of singularities and stable curves (a loop degenerates and becomes a double point), see \cite{BCGGM}.\newline

We do not have such quantities for branched affine structures. However, there is function $\theta$ that is the maximal angle of a hyperbolic cylinder embedded in the surface. Theorem 2.8 initially proved by Veech proves that triangulability is equivalent to condition $\theta < \pi$. It is an open problem stated in \cite{DFG} to characterize dilation surface with $\theta=0$. Does this condition characterizes translation surfaces ?\newline

Function $\theta$ clearly controls one kind of degeneracy of triangulated branched affine structures: when a hyperbolic cylinder becomes too long to be triangulated. Here, we introduce quantity $\alpha$ that should control every kind of degeneracy. It is defined through a minimax process.

\begin{defn}
For a triangulated branched affine surface $X$ and one of its geometric triangulations $T \in \mathcal{T}$, we define $\alpha_{X,T}$ to be the smallest angle for a triangle of $T$. Then, we define $\alpha_{X}=\underset{T\in \mathcal{T}}\sup(\alpha_{X,T})$.
\end{defn}

In any geometric triangulation, the smallest angle is at most $\frac{\pi}{3}$, surfaces tiled by equilateral triangles satisfy $\alpha=\frac{\pi}{3}$. They have been studied in \cite{BG} as extremal exemples for systolic inequalities in translation surfaces.

\begin{prop}
For any topological surface of genus $g$ with $n$ marked points and any $t>0$, $\left\{ X \in T_{g,n}~|\alpha_{X} \geq t  \right\}$ is compact.
\end{prop}

\begin{proof}
Up to isomorphism of their incidence structures, there are finitely many topological triangulations on a topological surface of genus $g$ and $n$ marked points.
For any sequence of branched affine surfaces of $T_{g,n}$, there is a isomorphism type of topological trangulations such that there is a subsequence of branched affine structures for which at least one geometric triangulation realizes this incidence type. Then, the subset of the parameter space of a affine triangles with angles bounded by below is compact. Therefore, for every triangle, there is an accumulation point that is a nondegenerate triangle. We glue every accumulation triangle following the incidence type and get an accumulation point the sequence in the moduli space. Function $\alpha$ is clearly continuous so the preimage of $[t,+\infty[$ is a closed set.
\end{proof}

Degeneration controlled by $\theta$ is also controlled by $\alpha$

\begin{prop}
For any triangulable branched affine surface, the angle of any embedded hyperbolic cylinder is at most $\pi-\alpha$.
\end{prop}

\begin{proof}
For any hyperbolic cylinder of angle $\beta$ and any geometric triangulation of $X$, a generic point of the interior of the cylinder belongs to the interior of a triangle such that at least two sides are trajectories that cross both boundaries of the cylinder. We can define directions that are not ambiguous in the union of the triangle and the cylinder (linear holonomy inside the cylinder preserves the directions and the triangle is simply connected). None of these two cylinders is parallel to a closed geodesic of the cylinder (because it should cross each of them). Therefore, their direction belongs to an interval whose length is $\pi-\beta$. Therefore, the angle between these two sides is at most $\pi-\beta$ and at least $\alpha$. Consequently, we have $\beta \leq \pi -\alpha$.\newline
\end{proof}

\section{On compactifications of spaces of branched affine surfaces}

We first consider few examples of degeneration of branched affine surfaces and possible compactifications.

\subsection{Some strata of branched affine surfaces on spheres with four singularities}

We consider branched affine surfaces of genus zero with four conical singularities of angles $\theta_{1},\dots,\theta_{4}$ and dilation ratios $\lambda_{1},\dots,\lambda_{4}$. We have $\sum \theta_{i} = 4\pi$. We consider strata of branched affine structures where the numbers are fixed.\newline
In strata where we have $\theta_{1},\theta_{2},\theta_{3} < \pi$ and $\theta_{4}>\pi$, an easy application of Lemma 2.8 shows that any surface is star-shaped with three self-folded triangles around a small singularity (see Figure 1) and a central triangle (possibly degenerate). Its parametrization is quite easy because the affine type of each of the three self-folded triangles is completely characterized by the conical singularity (an angle and a ratio of lengths). A branched affine structure is thus characterized by the affine type of the central triangle. If this triangles becomes flat, then a vertex appears in the boundary of one self-folded triangle and if we push the deformation further, another central triangle appears inside this triangle. Therefore, the locus where surfaces are star-shaped is connected. There are three singularities in the moduli space, corresponding to the degeneracy of each of the three sides of the central triangle. The connected component is a sphere with three punctures.\newline

\begin{figure}
\includegraphics[scale=0.3]{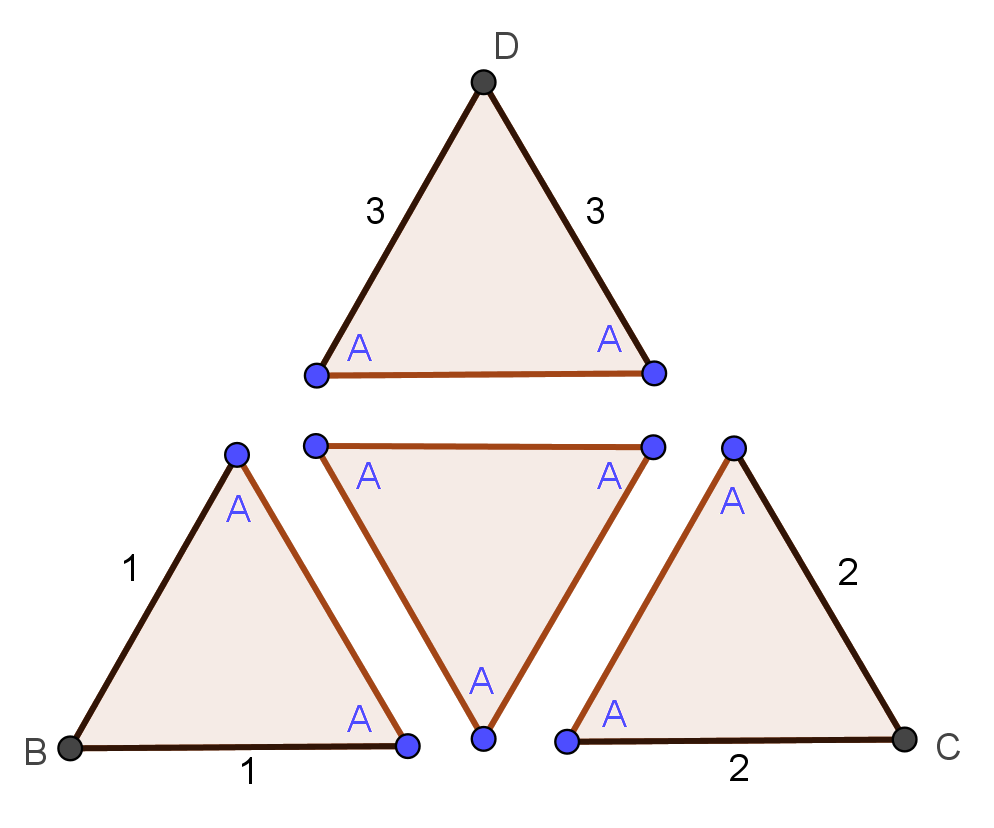}
\caption{Star-shaped branched affine sphere with four singularities}
\end{figure}

We can compactify the moduli space by adding a degenerated object at every puncture of the stratum. Without loss of generality, we consider the shrinking of the side of the central triangle bounding the self-folded triangle around the singularity of angle $\theta_{1}$. The surface is disconnected into two components. One is the union of the self-folded triangles around $\theta_{2}$ and $\theta_{3}$. The other is the self-folded triangle around $\theta_{1}$. The problem is that the boundary of this triangle is identified with a point (the conical singularity $A$, see figure 2). The solution is to take this triangle as infinite. The directions at the infinity can be interpreted as a singularity of angle $-\theta_{1}$, just like poles in translation surfaces (see \cite{Ta1}). The triangle becomes a sphere with two singularities and we can identify the pole with the rest of the singularity $A$ in the rest of the surface. We get a branched affine structure on a stable curve which is a the union of two sphere glued on a double point. We remark that the sum of the angles for this singularity is $-\theta_{1}+(2\pi-\theta_{2}-\theta_{3})=\theta_{4}-2\pi$.
\newline

\begin{figure}
\includegraphics[scale=0.3]{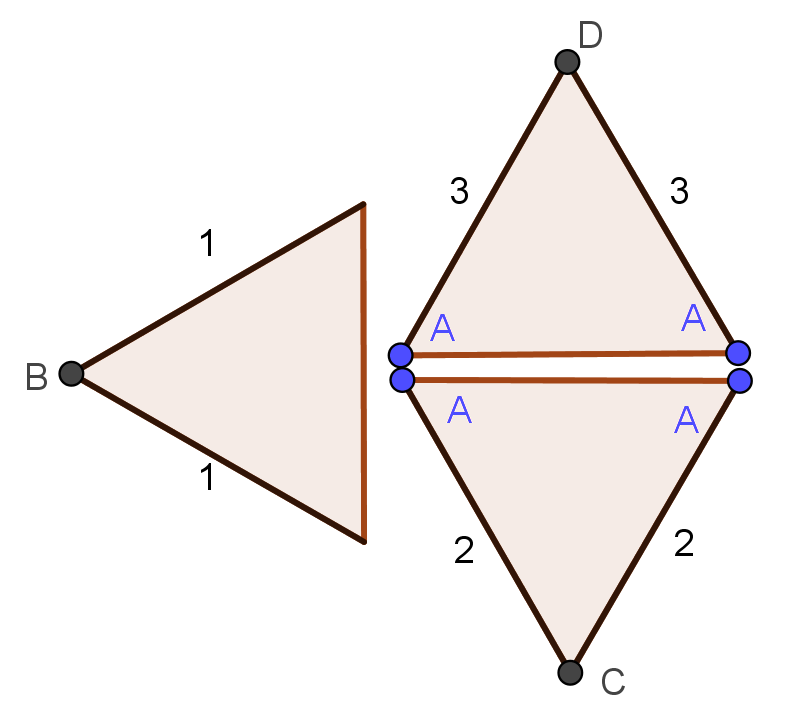}
\caption{Degeneracy of the branched affine structure}
\end{figure}

\subsection{Degeneration of Branched affine tori}

We consider a family of branched affine tori formed by a portion of annulus of angle $\theta<\pi$ with a dilation ratio of $\lambda$ with a marked point. We are looking for an interpretation of the limits $\lambda \to 1$ and $\lambda \to +\infty$.\newline
Annuli with angles smaller than $\pi$ can be triangulated as trapezoids. If the dilation ratio between their basis tends to $1$, without changing the angle between the lateral sides, then the basis become infinitely long. Just like in the compactification of translation surfaces, infinitely long cylinders can be pinched to get a pair of simple poles (corresponding to singularities of angle zero). Thus, the limit object of the torus is a pinched torus, that is a sphere with two identified points. However, we lost the information about angle $\theta$.\newline

When the dilation ratio tends to infinity, the picture is different. One of the basis of the trapezoid shrinks and a point is identified with a side. We get a self-folded triangle around a singularity of angle $\theta$. The solution is, just like degeneracy of star-shaped branched affine spheres, to set the triangle to be infinite and replace it by an infinite cone of angle $\theta$. The directions at infinity are interpreted as a singularity of angle $-\theta$. We get a sphere with two singularities that are then glued on each other. It is also a pinched torus.

\subsection{Some remarks}

The few examples we considered show that we clearly need to expand the scope to singularities of pole type. Indeed, poles of arbitrary order (arbitrary angle here) appear in \cite{BCGGM} in the compactification of strata of Abelian differentials. It would be interesting to know if we can get a meaningful compactification of branched affine structures with only two new ingredients: poles and stable curves.\newline

Poles need further geometric work because neighborhoods of these singularities are not triangulable. The notion of core (see \cite{HKK,Ta1}) would have to be generalized to the case of branched affine structures with these new singularities. Classically, the core is the convex hull of the conical singularities and is equal to the triangulated locus in the case of branched affine structures with only conical singularities.\newline

We can also remark that in the example of the sphere with four singularities, the degeneration never corresponds to the collision of two singularities.
\newline

\textit{Acknowledgements.} I thank Selim Ghazouani, Dmitry Novikov and Pierre Villemot for their valuable remarks. The author is supported by a fellowship of Weizmann Institute of Science. This research was supported by the Israel Science Foundation (grant No. 1167/17).\newline

\nopagebreak
\vskip.5cm
\end{document}